\documentclass[amsart,11pt]{article}
\usepackage{color,graphicx,latexsym,amssymb,amsmath,amsfonts,amsthm,float}
\usepackage{url}
\usepackage[bookmarks=false]{hyperref}
\hypersetup{colorlinks}
\usepackage{authblk}

\newtheorem{theorem}{Theorem}
\newtheorem{lemma}[theorem]{Lemma}

\theoremstyle{definition}
\newtheorem*{remark}{Remark}

\newcommand{\Z}{\mathbb Z}
\newcommand{\E}{\mathbb E}
\newcommand{\PP}{\mathbb P}
\newcommand{\B}{\mathcal{B}}

\definecolor{db}{rgb}{0.1,0,0.75}
\definecolor{lm}{cmyk}{0 ,1,0,0}

\def\aligned{\mathbf{\uparrow}}
\def\recurrent{\rho_{\mathrm{rec}}}
\def\div{\mathrm{div}\,}
\def\col{\mathrm{col}}

\newcommand{\N}{\mathbb N}
\newcommand{\old}[1]{}

\title{Escape rates for rotor walks in $\mathbb{Z}^d$}

\author[1]{Laura Florescu\thanks{florescu@cims.nyu.edu}}
\author[2]{Shirshendu Ganguly\thanks{sganguly@uw.edu}}
\author[3]{Lionel Levine\thanks{\url{http://www.math.cornell.edu/\~levine}. Partially supported by NSF grant DMS-1243606.}}
\author[4]{Yuval Peres\thanks{peres@microsoft.com}}
\affil[1]{New York University}
\affil[2]{University of Washington}
\affil[3]{Cornell University}
\affil[4]{Microsoft Research} 
\date{August 2, 2013}

\begin{document}
\maketitle


\begin{abstract}
Rotor walk is a deterministic analogue of random walk. We study its recurrence and transience properties on $\Z^d$ for the initial configuration of all rotors aligned.  If $n$ particles in turn perform rotor walks starting from the origin, we show that the number that escape (i.e., never return to the origin) is of order $n$ in dimensions $d \geq 3$, and of order $n/\log n$ in dimension $2$.
\end{abstract}

\section{Introduction}
In a \emph{rotor walk} on a graph, the successive exits from each vertex follow a prescribed periodic sequence.  For instance, in the square grid $\Z^2$, successive exits could repeatedly cycle through the sequence North, East, South West.  Such walks were first studied in \cite{WLB96} as a model of mobile agents exploring a territory, and in \cite{priezzhev} as a model of self-organized criticality. 
In a lecture at Microsoft in 2003 \cite{propp}, Jim Propp proposed rotor walk as a deterministic analogue of random walk, which naturally invited the question of whether rotor walk is recurrent in dimension $2$ and transient in dimensions $3$ and higher.
One direction was settled immediately by Oded Schramm, who showed that rotor walk is ``at least as recurrent'' as random walk.  Schramm's elegant argument, which we recall below, applies to any initial rotor configuration $\rho$.  

The other direction is more subtle because it depends on $\rho$. We say that $\rho$ is \emph{recurrent} if the rotor walk started at the origin with initial configuration $\rho$ returns to the origin infinitely often; otherwise, we say that $\rho$ is \emph{transient}. 
Angel and Holroyd \cite{angel-holroyd-rec} showed that for all $d$ there exist initial rotor configurations on $\Z^d$ such that rotor walk is recurrent.  These special configurations are primed to send particles initially back toward the origin.
The purpose of this note is to analyze the case $\rho=\aligned$ when all rotors send their first particle in the same direction.  To measure how transient this configuration is, we run $n$ rotor walks starting from the origin and record whether each returns to the origin or escapes to infinity.  We show that the number of escapes is of order $n$ in dimensions $d\geq 3$, and of order $n/\log n$ in dimension $2$.

To give the formal definition of a rotor walk, write $\mathcal{E} = \{\pm e_1, \ldots, \pm e_d\}$ for the set of $2d$ cardinal directions in $\Z^d$, and let $\mathcal{C}$ be the set of cyclic permutations of $\mathcal{E}$.
A \emph{rotor mechanism} is a map $m: \Z^d \to \mathcal{C}$, and a \emph{rotor configuration} is a map $\rho: \Z^d \to \mathcal{E}$.  A \emph{rotor walk} started at $x_0$ with initial configuration $\rho$ is a sequence of vertices $x_0, x_1, \ldots \in \Z^d$ and rotor configurations $\rho=\rho_0, \rho_1, \ldots$ such that for all $n\geq 0$
	\[ x_{n+1} = x_n + \rho_n(x_n). \]
and
	\[ \rho_{n+1}(x_n) = m(x_n)(\rho_n(x_n)) \]
and $\rho_{n+1}(x) = \rho_n(x)$ for all $x \neq x_n$. 

For example in $\mathbb{Z}^2$, each rotor $\rho(x)$ points North, South,
East or West.  An example of a rotor mechanism is the permutation North $\mapsto$ East $\mapsto$ South $\mapsto$ West $\mapsto$ North at all $x \in \Z^2$.  The resulting rotor walk in $\Z^2$ has the following description: A particle repeatedly steps in the direction indicated by the rotor at its current location, and then this rotor turns $90$ degrees clockwise.
Note that this ``prospective'' convention --- move the particle before updating the rotor --- differs from the ``retrospective'' convention of past works such as \cite{angel-holroyd-rec,HLMPPW}.  In the prospective convention, $\rho(x)$ indicates where the next particle will step from $x$, instead of where the previous particle stepped.  The prospective convention is often more convenient when studying questions of recurrence and transience.

In this paper we fix once and for all a rotor mechanism $m$ on $\Z^d$. Now depending on the initial rotor configuration $\rho$, one of two things can happen to a rotor walk started from the origin:
\begin{enumerate}
\item The walk eventually returns to the origin; or
\item The walk never returns to the origin, and visits each vertex in $\Z^d$ only finitely often. 
\end{enumerate} 
Indeed, if any site were visited infinitely often, then each of its neighbors must be visited infinitely often, and so the origin itself would be visited infinitely often.  In case 2 we say that the walk ``escapes to infinity.''  Note that after the walk has either returned to the origin or escaped to infinity, the rotors are in a new configuration.

To quantify the degree of transience of an initial configuration $\rho$, consider the following experiment: 
let each of $n$ particles in turn perform rotor walk starting from the origin until either returning to the origin or escaping to infinity.  Denote by $I(\rho,n)$ the number of walks that escape to infinity.  (Importantly, we do not reset the rotors in between trials!) 

Schramm \cite{schramm} proved that for any $\rho$,
	\begin{equation} \label{schramm} \limsup_{n \to \infty} \frac{I(\rho,n)}{n} \leq \alpha_d \end{equation}
where $\alpha_d$ is the probability that simple random walk in $\Z^d$ does not return to the origin.  Our first result gives a corresponding lower bound for the initial configuration $\aligned$ in which all rotors start pointing in the same direction: $\aligned(x) = e_d$ for all $x \in \Z^d$.

\begin{theorem}
\label{thm:transient}
For the rotor walk on $\mathbb{Z}^d$ with $d\ge3$ with all rotors initially aligned $\aligned$, a positive fraction of particles escape to infinity; that is,
	 \[ \liminf_{n \to \infty} \frac{I(\aligned,n)}{n} > 0. \]
\end{theorem}

One cannot hope for such a result to hold for an arbitrary $\rho$: Angel and Holroyd \cite{angel-holroyd-rec} prove that in all dimensions there exist rotor configurations $\recurrent$ such that $I(\recurrent,n)=0$ for all $n$. Reddy first proposed such a configuration in dimension $3$ on the basis of numerical simulations \cite{Reddy}.

Our next result concerns the fraction of particles that escape in dimension $2$: for any rotor configuration $\rho$ this fraction is at most $\frac{\pi/2}{\log n}$, and for the initial configuration $\aligned$ it is at least $\frac{c}{\log n}$ for some $c>0$.

\begin{theorem} 
\label{thm:logfraction}
For rotor walk in $\Z^2$ with any rotor configuration $\rho$, we have
	\[ \limsup_{n \to \infty} \frac{I(\rho,n)}{n /\log n} \leq \frac{\pi}{2}. \]
Moreover, if all rotors are initially aligned $\aligned$, then
	\[ \liminf_{n \to \infty} \frac{I(\aligned,n)}{n / \log n} >0. \] 
\end{theorem}

\begin{figure}[H]
\centering
\includegraphics[scale=.6]{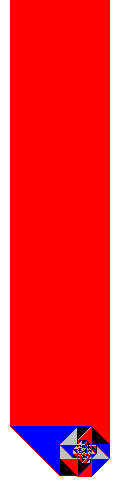}
\includegraphics[scale=.6]{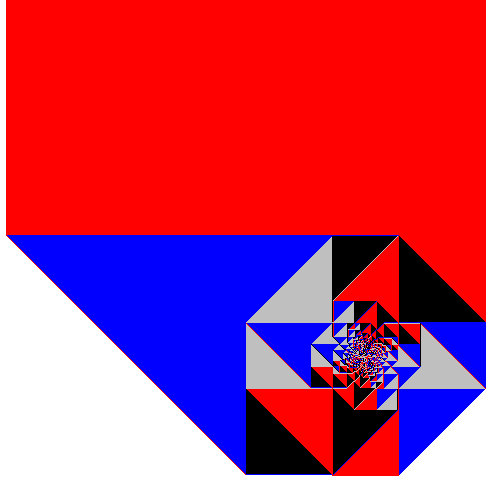}
\caption{The configuration of rotors in $\Z^2$ after $n$ particles started at the origin have escaped to infinity, with initial configuration $\uparrow$ (that is, all rotors send their first particle North).  Left: $n=100$; Right: $n=480$.  Each non-white pixel represents a point in $\Z^2$ that was visited at least once, and its color indicates the direction of its rotor.
\label{f.finalrotors}}
\end{figure}

\section{Schramm's argument}

One way to estimate the number of escapes to infinity of a rotor walk is to look at how many particles exit a large ball before returning to the origin.  Let
	\[ \mathcal{B}_r=\{x\in \mathbb{Z}^d: |x| < r\}\]
be the set of lattice points in the open ball of radius $r$ centered at the origin. Here $|x|=(x_1^2+ \cdots + x_d^2)^{1/2}$ denotes the Euclidean norm of $x$.
Consider rotor walk started from the origin and stopped on hitting the boundary
	\[ \partial\mathcal{B}_r = \{y \in \Z^d \,:\, y\notin \B_r \mbox{ and } y\sim x \mbox{ for some } x \in \B_r \}. \]
Since $\mathcal{B}_r$ is a finite connected graph, this walk stops in finitely many steps.  

Starting from initial rotor configuration $\rho$, let each of $n$ particles in turn perform rotor walk starting from the origin until either returning to the origin or exiting the ball $\B_r$.
Denote by $I_r(\rho,n)$ the number of particles that exit $\B_r$.  The next lemmas give convergence and monotonicity of this quantity.

\begin{lemma}
\label{finiteball}\cite[Lemma 18]{holroyd-propp}  For any rotor configuration $\rho$ and any $n \in \N$, we have $I_r(\rho,n) \rightarrow I(\rho,n)$ as $r \rightarrow \infty$.
\end{lemma}

\begin{proof}
Let $w_n(x)$ be the number of exits from $x$ by $n$ rotor walks started at $o$ and stopped if they return to $o$.  Then $I(\rho,n)$ is determined by the values $w_n(x)$ for neighbors $x$ of $o$.

Let $w_n^r(x)$ be the number of exits from $x$ by $n$ rotor walks started at $o$ and stopped on hitting $\partial \mathcal{B}_r \cup \{o\}$. Then $I_r(\rho,n)$ is determined by the values $w_n^r(x)$ for neighbors $x$ of $o$.

We first show that $w_n^r \leq w_n$ pointwise. Let $w_n^{r,t}(y)$ be the number of exits from $y$ before time $t$ if the walks are stopped on hitting $\partial \mathcal{B}_r \cup \{o\}$.  If $w_n^r \not \leq w_n$, then choose $t$ minimal such that $w_n^{r,t} \not \leq w_n$.  Then there is a single point $y$ such that $w_n^{r,t}(y) > w_n(y)$.  Note that $y \neq o$, because $w_n^r(o) = w_n(o) = n$.
Since $w_n^{r,t}(x) \leq w_n(x)$ for all $x \neq y$, at time $t$ in the finite experiment the site $y$ has received at most as many particles as it ever receives in the infinite experiment. But $y$ has emitted strictly more particles in the finite experiment than it ever emits in the infinite experiment, so the number of particles at $y$ at time $t$ is $<0$, a contradiction.

Now we induct on $n$ to show that $w_n^r \uparrow w_n$ pointwise.  Assume that $w_{n-1}^r \uparrow w_{n-1}$.  Fix $s>0$.  There exists $R=R(s)$ such that $w_{n-1}^{r} = w_{n-1}$ on $\mathcal{B}_s$ for all $r \geq R$.
If the $n$th walk returns to $o$ then it does so without exiting $\mathcal{B}_{S}$ for some $S$; in this case $w_n^r = w_n$ on $\mathcal{B}_s$ for all $r \geq \max(R,S)$.

If the $n$th walk escapes to infinity, then there is some radius $S$ such that after exiting $\mathcal{B}_S$ the walk never returns to $\mathcal{B}_s$.  Now choose $R'$ such that $w_{n-1}^{R'} = w_{n-1}$ on $\mathcal{B}_S$.  Then we claim $w_n^r = w_n$ on $\mathcal{B}_s$ for all $r \geq R'$.  Denote by $\rho_{n-1}^r$ (resp.\ $\rho_{n-1}$) the rotor configuration after $n-1$ walks started at the origin have stopped on $\partial \mathcal{B}_r \cup \{o\}$ (resp.\ stopped if they return to $o$).
If $r \geq R'$ then the rotor walks started at $o$ with initial conditions $\rho_{n-1}^{r}$ and $\rho_{n-1}$ agree until they exit $\mathcal{B}_S$.  Thereafter the latter walk never returns to $\mathcal{B}_s$, hence $w_n^r \geq w_n$ on $\mathcal{B}_s$.  Since also $w_n^r \leq w_n$ everywhere, the inductive step is complete.
\end{proof}

\old{ 
\begin{proof}
Let $r_1$ be such that all particles which
return to the origin do not leave $\mathcal{B}_{r_1}$, and let $r_2$
be such that all particles which escape to infinity never return to
$\mathcal{B}_{r_1}$ after they leave $\mathcal{B}_{r_2}$. 
If $r \geq r_2$, then stopping particles when they hit $\partial \B_r$ has no effect on the rotors inside $\B_{r_1}$, and therefore has no effect on whether future particles return to the origin.  Hence
$I_{r}(\rho,n) = I(\rho,n)$ for all $r \geq r_2$.
\end{proof}
}

For the next lemma we recall the \emph{abelian property} of rotor walk \cite[Lemma~3.9]{HLMPPW}.  Let $A$ be a finite subset of $\Z^d$. In an experiment of the form ``run $n$ rotor walks from prescribed starting points until they exit $A$,'' suppose that we repeatedly choose a particle in $A$ and ask it to take a rotor walk step.  Regardless of our choices, all particles will exit $A$ in finitely many steps; for each $x \in A^c$, the number of particles that stop at $x$ does not depend on the choices; and for each $x \in A$, the number of times we pointed to a particle at $x$ does not depend on the choices.

\begin{lemma}
\label{monotone}\cite[Lemma 19]{holroyd-propp}
For any rotor configuration $\rho$, any $n \in \N$ and any $r<R$, we have $I_{R}(\rho,n)\le I_r(\rho,n)$.
\end{lemma}

\begin{proof}
By the abelian property, we may compute $I_{R}(\rho,n)$ in two stages. First stop particles when they reach $\partial \B_r \cup \{o\}$, where $o \in \Z^d$ is the origin, and then let the $I_r(\rho,n)$ particles stopped on $\partial \B_r$ continue walking until they reach $\partial \B_{R} \cup \{o\}$.  Therefore at most $I_r(\rho,n)$ particles stop in $\partial \B_R$.
\end{proof}


Oded Schramm's upper bound \eqref{schramm} begins with the observation that if $2dm$ particles at a single site $x \in \Z^d$ each take a single rotor walk step, the result will be that $m$ particles move to each of the $2d$ neighbors of $x$.  Fix $r,m \in \N$ and consider $N = (2d)^r m$ particles at the origin.  Let each particle take a single rotor walk step.  Then repeat $r-1$ times the following operation: let each particle that is not at the origin take a single rotor walk step. 
The result is that for each path $(\gamma_0,\ldots,\gamma_\ell)$ of length $\ell \leq r$ with $\gamma_0=\gamma_\ell=o$ and $\gamma_i \neq o$ for all $1 \leq i \leq \ell-1$, exactly $(2d)^{-\ell}N$ particles traverse this path.  Denoting the set of such paths by $\Gamma(r)$ and the length of a path $\gamma$ by $|\gamma|$, the number of particles now at the origin is
	\[ N \sum_{\gamma \in \Gamma(r)} (2d)^{-|\gamma |} = N p  \]
where $p = \PP(T_o^+ \leq r)$ is the probability that simple random walk returns to the origin by time $r$.

Now letting each particle that is not at the origin continue performing rotor walk until hitting $\partial \B_r \cup \{o\}$, the number of particles that stop in $\partial \B_r$ is at most $N(1-p)$, so
	\[ \frac{I_r(\rho,N)}{N} \leq 1-p. \]
This holds for every $N$ which is an integer multiple of $(2d)^r$.  For general $n$, let $N$ be the smallest multiple of $(2d)^r$ that is $\geq n$.  Then
	\[ \frac{I_r(\rho,n)}{n} \leq \frac{I_r(\rho,N)}{N-(2d)^r} \]
The right side is at most $(1-p)(1 + 2(2d)^r/N)$, so
	\[ \limsup_{n \to \infty}  \frac{I(\rho,n)}{n} \leq \limsup_{n \to \infty} \frac{I_r(\rho,n)}{n} \leq 1-p = \PP(T_o^+ > r). \]
As $r \to \infty$ the right side converges to $\alpha_d$, completing the proof of \eqref{schramm}.

See Holroyd and Propp \cite[Theorem~10]{holroyd-propp} for an extension of
Schramm's argument to a general irreducible Markov chain with rational transition probabilities.

\section{An odometer estimate for balls in all dimensions}

To estimate $I_r(\rho,n)$, consider now a slightly different experiment.  Let each of $n$ particles started at the origin perform rotor walk until hitting $\partial \B_r$. (The difference is that we do not stop the particles on returning to the origin!) 
Define the \emph{odometer function} $u^{r}_{n}$ by
$$u^{r}_{n}(x) = \mbox {total number of exits from }x \mbox{ by }n \mbox{ rotor walks stopped on hitting }\partial\mathcal{B}_r. $$
Note that $u^r_n(x)$ counts the total number of exits (as opposed to the net number).

Now we relate the two experiments. 

\begin{lemma}
\label{inn}
For any $r>0$ and $n \in \N$ and any initial rotor configuration $\rho$, we have \[ I_r(\rho,u^r_{n}(o)) = n. \]
\end{lemma}
\begin{proof}
Starting with $N = u^r_n(o)$ particles at the origin, consider the following two experiments:
\begin{enumerate}
\item Let $n$ of the particles in turn perform rotor walk until hitting $\partial \mathcal{B}_r$.  

\item Let $N$ of the particles in turn perform rotor walk until hitting $\partial \mathcal{B}_r \cup \{o\}$.
\end{enumerate}
By the definition of $u^r_n$, in the first experiment the total number of exits from the origin is exactly $N$.  Therefore the two experiments have exactly the same outcome: $n$ particles reach $\partial \B_r$ and $N-n$  remain at the origin.
%
\end{proof}

Our next task is to estimate $u_n^r$. 
 We begin by introducing some notation.
Given a function $f$ on $\mathbb{Z}^d$, its \emph{gradient} is the function on directed edges given by
$$\nabla f(x,y) :=f(y)-f(x).$$
Given a function $\kappa$ on directed edges of $\Z^d$, its \emph{divergence} is the function on vertices given by
$$\div \kappa(x):=\frac{1}{2d}\sum_{y \sim x} \kappa(x,y)$$
where the sum is over the $2d$ nearest neighbors of $x$. 
The \emph{discrete Laplacian} of $f$ is the function $$\Delta f(x):= \div (\nabla
f)(x) =\frac{1}{2d}\sum_{y \sim x}f(y)- f(x).$$

We recall some results from \cite{levine-peres}. 

\begin{lemma}\cite[Lemma 5.1]{levine-peres}
\label{lemma-k}
For a directed edge $(x,y)$ in $\mathbb{Z}^d$, denote by $\kappa(x,y)$ the net number of crossings from $x$ to $y$ by $n$ rotor walks started at the origin and stopped on exiting $\B_r$. Then
$$\nabla u_n^r(x,y)= -2d\, \kappa(x,y) + R(x,y)$$ for some edge function $R$
  satisfying $|R(x,y)| \le 4d-2$ for all edges $(x,y)$.
\end{lemma}


Denote by $(X_j)_{j \geq 0}$ the simple random walk in $\Z^d$, whose increments are independent and uniformly distributed on $\mathcal{E} = \{ \pm e_1, \ldots, \pm e_d \}$.
Let $T = \min \{j \,:\, X_j \not\in \B_r\}$ be the first exit time from the ball of radius $r$.  For $x,y \in \B_r$, let
 	\[ G_r(x,y) ={\E}_x \#  \{j<T|X_j=y\} \] 
be the expected number of visits to $y$ by a simple random walk started at $x$ before time $T$.
The following well known estimates can be found in \cite[Prop.\ 1.5.9, Prop.\ 1.6.7]{lawler}: for a constant $a_d$ depending only on $d$,
\begin{equation}
\label{green-estimates}
G_r(x,o) = \begin{cases}
 a_d(|x|^{2-d} - r^{2-d}) + O (|x|^{1-d}), &d\ge 3 \\
 \frac{2}{\pi}(\log r-\log|x|) +O(|x|^{-1}), &d=2.
       \end{cases}
\end{equation}
We will also use \cite[Theorem 1.6.6]{lawler} the fact that in dimension $2$,
\begin{equation}
\label{green-origin}
G_r(o,o) = \frac{2}{\pi} \log r + O(1).
\end{equation}
(As usual, we write $f(n) = \Theta(g(n))$ (respectively, $f(n) = O(g(n))$) to mean that there is a constant $0<C<\infty$ such that $1/C < f(n)/g(n) < C$ (respectively, $f(n)/g(n)<C$) for all sufficiently large $n$.
Here and in what follows, the constants implied in $O()$ and $\Theta()$ notation depend only on the dimension~$d$.)


The next lemma bounds the $L^1$ norm of the discrete gradient of the function $G_r(x, \cdot)$. It appears in \cite[Lemma 5.6]{levine-peres} with the factor of $2$ omitted (this factor is needed for $x$ close to the origin). The proof given there actually shows the following.

\begin{lemma}
\label{greens}
Let $x \in \mathcal{B}_r$ and let $ \rho  =r+1-|x|$. Then for some $C$ depending only on $d$,
$$\sum_{y \in \mathcal{B}_r}\sum_{z\sim y} |G_r(x,y)-G_r(x,z)|\le C \rho \log \frac{2r}{\rho}.$$
\end{lemma}

The next lemma is proved in the same way as the inner estimate of \cite[Theorem~1.1]{levine-peres}. 
Let $f(x) = nG_r(x,o)$.
\begin{lemma} 
\label{positive}
In $\mathbb{Z}^d$, let $x \in \mathcal{B}_r$ and $ \rho  =r+1-|x|$. Then,
$$|u^r_{n}(x)-f(x)|\le  C \rho \log\frac{2r}{\rho}+4d.$$ where $u^r_{n}$ is the odometer function for $n$ particles performing rotor walk stopped on exiting $\mathcal{B}_r$, and $C$ is the constant in Lemma~\ref{greens}.
\end{lemma} 

\begin{proof}
If we consider the rotor walk stopped on exiting
$\mathcal{B}_r$, all sites that have positive odometer value have been hit by
particles. Using notation of Lemma~\ref{lemma-k}, we notice that since the net number of particles to enter a site $x\neq o$ not on the 
boundary is zero, we have $2d\, \div \kappa (x)= 0$. For the origin,
$2d\, \div \kappa (o)= n$. Also, the odometer function vanishes on
the boundary, since the boundary does not emit any particles.

Write $u = u^r_n$. Using the definition of $\kappa$ in Lemma~\ref{lemma-k}, we see that 
\begin{eqnarray}\label{eq:lap1}
\Delta u(x) &=& \div R(x), \,\,x \ne o,\\
\Delta u(o) &=& -n + \div R(o). \label{eq:lap2}
\end{eqnarray}
Then $\Delta f(x)=\, 0 $ for $x \in \mathcal{B}_r \setminus \{o\}$ and $\Delta f(o)=\, -n $ and $f$ vanishes on $\partial \mathcal{B}_r$.

Since $u(X_T)$ is equal to $0$, we have 
	\[ u(x)=\sum_{k\ge0}\E_x(u(X_{k\wedge T})- u(X_{(k+1)\wedge T})). \]
Also, since the $k^{th}$ term in the sum is zero when $T\le k$
	\[ \E_x(u(X_{k\wedge T})- u(X_{(k+1)\wedge T})|\mathcal{F}_{k\wedge T})= -\Delta u(X_{k}){1}_{\{T>k\}} \]
where $\mathcal{F}_j = \sigma(X_0,\ldots,X_j)$ is the standard filtration for the random walk.
 
Taking expectation of the conditional expectations and using (\ref{eq:lap1}) and (\ref{eq:lap2}), we get 
\begin{eqnarray*}
u(x)&=& \sum_{k\ge0}\E_x\left[1_{\{T>k\}}(n1_{\{X_k=o\}}-\div R(X_k) )\right]\\
&=& n\,{\E}_x \#  \{k<T|X_k=0\}-\sum_{k\ge0}\E_x\left[1_{\{T>k\}} \div R(X_k) \right].
\end{eqnarray*}
So,
$$
u(x)-f(x)=-\frac{1}{2d}\sum_{k\ge0}\E_x\left[1_{\{T>k\}}\sum_{z\sim X_k}\,R(X_k,z) \right].$$
 Let $N(y)$ be the number of edges joining $y$ to $\partial \mathcal{B}_r$. Since $\E_x \sum_{k\geq 0} 1_{\{T>k\}} N(X_k) = 2d$, and $|R| \leq 4d$, the terms with $z\in \partial \mathcal{B}_r$ contribute at most $8d^2$ to the sum.
Thus,
\begin{equation}
\label{uppbound}
|u(x)-f(x)|\le \frac{1}{2d} \left| \sum_{k\ge0}\E_x\left[\sum_{\substack{ y,z \in \mathcal{B}_r \\y \sim z }}1_{\{T>k\}\cap \{X_k=y\}}\,R(y,z) \right]\right|+4d.
\end{equation}

Note that for $y \in \mathcal{B}_r$ we have $\{X_k=y\}\cap \{T>k\}=\{X_{k\wedge T}=y\}$. Considering $p_k(y)=\mathbb{P}_x (X_{k\wedge T}=y)$, and noting that $R$ is antisymmetric (because of antisymmetry in Lemma~\ref{lemma-k}), we see that 
\begin{eqnarray*}\sum_{\substack{ y,z \in B_r \\y \sim z }}p_k(y)R(y,z)&= &-\sum_{\substack{ y,z \in B_r \\y \sim z }}p_k(z)R(y,z)\\
&=& \sum_{\substack{ y,z \in B_r \\y \sim z }}\frac{p_k(y)-p_k(z)}{2}R(y,z).
\end{eqnarray*}
Summing over $k$ in (\ref{uppbound}) and using the fact that $|R|\le 4d$, we conclude that 
$$|u(x)-f(x)|\le \sum_{\substack{ y,z \in B_r \\y \sim z }}  |G(x,y)-G(x,z)|+4d.$$ The result now follows from the estimate of the gradient of Green's function in Lemma~\ref{greens}.
\end{proof}

Now we make our choice of radius, $r=n^{1/(d-1)}$.
The next lemma shows that for this value of $r$, the support of the odometer function contains a large sphere.

\begin{lemma}
\label{support3}
There exists a sufficiently small $\beta>0$ depending only on $d$, such that for any initial rotor configuration and $r = n^{1/(d-1)}$ we have $u^r_n(x) > 0$ for all $x \in \partial \mathcal{B}_{\beta r}$.
\end{lemma}

\begin{proof}
For $x \in \partial \mathcal{B}_{\beta{r}}$ we have $\beta r\le|x|\le\beta r+1$.
By Lemma~\ref{positive} we have 
	\[ |u^r_n (x) - f(x)| \leq C'(1-\beta)r \log\frac{1}{1-\beta} \]
for a constant $C'$ depending only on $d$.  To lower bound $f(x)$ we use (\ref{green-estimates}): in dimensions $d\geq 3$ we have
 \begin{align*} f(x)=nG_r(x,o) &\geq n (a_d(|x|^{2-d}-r^{2-d}) - K|x|^{1-d})  \\
 	&=  a_d (\beta^{2-d}-1)n r^{2-d} - K\beta^{1-d} \end{align*}
for a constant $K$ depending only on $d$. Since $r= n r^{2-d}$, we can take $\beta>0$ sufficiently small so that
 	\[ a_d (\beta^{2-d}-1)n r^{2-d} - K\beta^{1-d} >  2C'(1-\beta)r \log\frac{1}{1-\beta} \] 
for all sufficiently large $n$.
Hence $u^r_n(x)>0$.

In dimension $2$, we have $r=n$ and $nG_n(x,o) \geq n\frac{2}{\pi}\log \frac{1}{\beta} - \frac{K}{\beta},$ by (\ref{green-estimates}). So
for $\beta$ small enough, we have
that  \[ nG_{n}(x,o)= n\frac{2}{\pi}\log\frac{1}{\beta} - \frac{K}{\beta} >
C'(1-\beta)n \log\frac{1}{1-\beta} \] 
for all sufficiently large $n$.
Hence $u^n_n(x)>0$.
\end{proof}

Identify $\Z^d$  with $\Z^{d-1} \times \Z$ and call each set of the form $(x_1, \ldots, x_{d-1}) \times \Z$ a ``column.''  Starting $n$ particles at the origin and letting them each perform rotor walk until exiting $\mathcal{B}_r$ where $r= n^{1/(d-1)}$, let $\col (\rho,n)$ be the number of distinct columns that are visited.  
 That is, 
$$\col(\rho,n) =\#\{(x_1,\ldots,x_{d-1}): u^{r}_n(x_1,x_2,\ldots,x_d)>0 \mbox{ for some }x_d \in \Z  \}.$$ 
By Lemma~\ref{support3}, every site of $\partial \B_{\beta r}$ is visited at least once, so  
	\begin{align} \col(\rho,n) &\geq \# \{(x_1,\ldots, x_{d-1}):\,(x_1,x_2,\ldots, x_d)\in \partial \mathcal{B}_{\beta r} \mbox{ for some } x_d \in \Z\} \nonumber \\
	&\geq C(\beta r)^{d-1} = \Theta(n). \label{manycolumns} \end{align} 

All results so far have not made any assumptions on the initial configuration. The next lemma assumes the initial rotor configuration to be $\aligned$.  The important property of this initial condition for us is that the first particle to visit a given column travels straight along that column in direction $e_d$ thereafter.

\begin{figure}[H]
\centering
\includegraphics[height=.3\textheight]{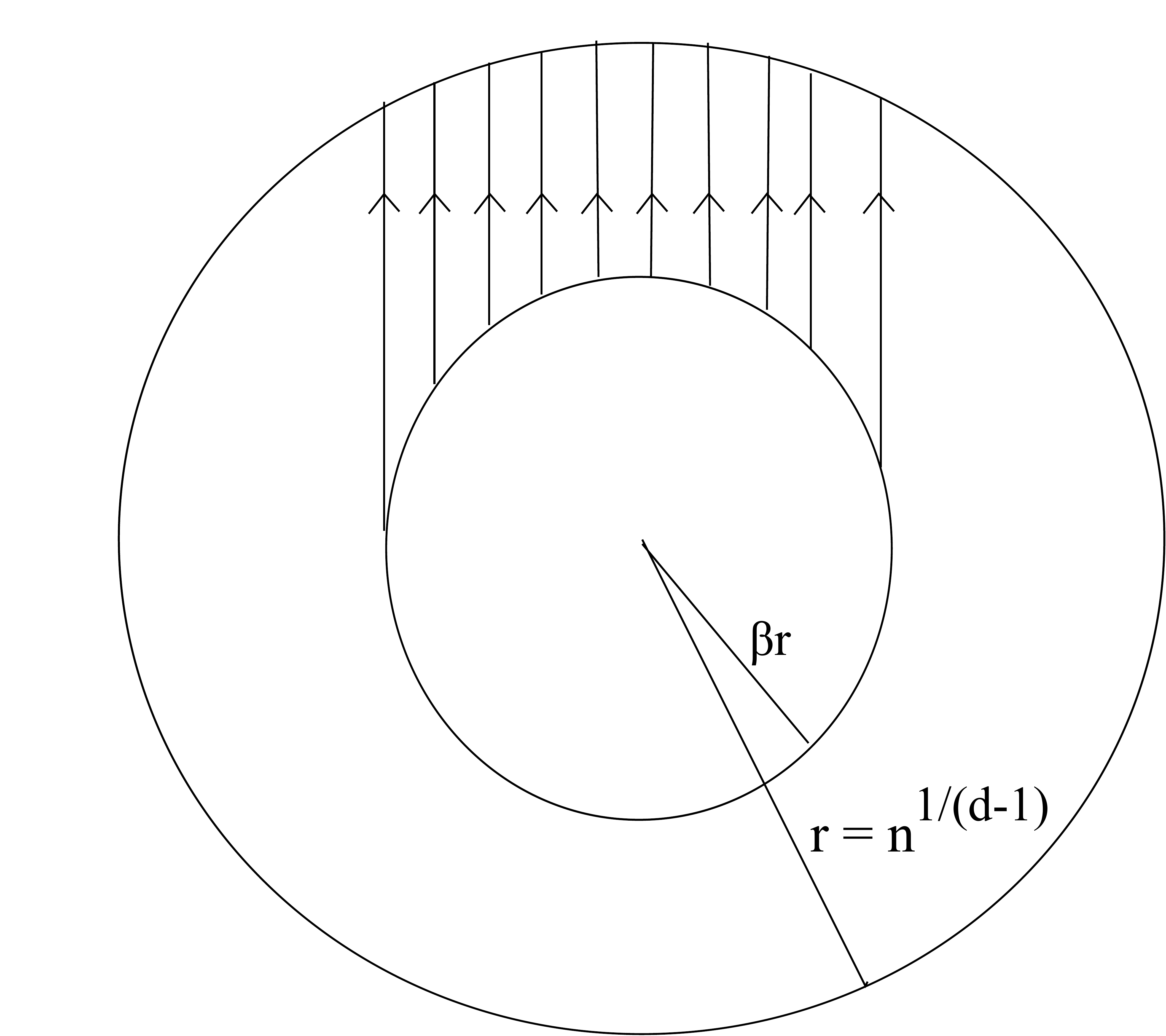}
\caption{
Diagram for the proof of Lemma~\ref{an<iun}. 
The first visit to each column results in an escape along that column, so at least $\col(\aligned,n)$ particles escape.
\label{figure2}}
\end{figure}

\begin{lemma}\label{an<iun}In $\mathbb{Z}^d$ with initial rotor configuration $\aligned$, we have \[ I_R(\aligned,u^r_n(o)) \ge \col(\aligned,n) \] 
for all $R\ge r$.
\end{lemma}
\begin{proof}
By the abelian property of rotor walk, we may compute $I_R(\rho,{u^r_n(o)})$ in two stages. First we stop the particles when they first hit $\partial \mathcal{B}_r \cup \{o\}$.  Then we let all the particles stopped on $\partial \mathcal{B}_r$ continue walking until they hit $\partial \mathcal{B}_R \cup \{o\}$.  By Lemma~\ref{inn}, exactly $n$ particles stop on $\partial \B_r$ during the first stage, and therefore $\col(\aligned,n)$ distinct columns are visited during the first stage.
 Because the initial rotors are $\aligned$, the first particle to visit a given column travels straight along that column to hit $\partial \B_R$ (Figure~\ref{figure2}).  Therefore the number of particles stopping in $\partial \B_R$ is at least $\col(\aligned,n)$.
\end{proof}

\section{The transient case: Proof of Theorem~\ref{thm:transient}}
\label{sec:transient}

In this section we consider $\mathbb{Z}^d$ for $d\ge 3$. 
We will prove Theorem~\ref{thm:transient} by comparing the number of escapes $I(\aligned,n)$ with $\col(\aligned,n)$.

Let $r=n^{1/(d-1)}$ and $N = u_n^r(o)$. By the transience of simple random walk in $\Z^d$ for $d\geq 3$ we have 
	\[ f(o) = nG_r(o,o) = \Theta(n). \]
By Lemma~\ref{positive} we have $|N-f(o)| = O(r)$ and hence $N=\Theta(n)$.
By Lemmas~\ref{finiteball} and~\ref{an<iun} we have $I(\aligned,N) \geq \col(\aligned, n)$.
Recalling \eqref{manycolumns} that $\col(\aligned,n) = \Theta(n)$ and that $I(\aligned,n)$ is nondecreasing in $n$, we conclude that there is a constant $c>0$ depending only on $d$ such that for all sufficiently large $n$
\[ \frac{I(\aligned,n)}{n}>c \]
which completes the proof.

\section{The recurrent case: Proof of Theorem~\texorpdfstring{\ref{thm:logfraction}}{2}}

In this section we work in $\Z^2$ and take $r=n$.
We start by estimating the odometer function at the origin for the rotor walk stopped on exiting $\mathcal{B}_n$.

\begin{lemma}
\label{greenorigin}
For any initial rotor configuration in $\Z^2$ we have
	\[ u_n^n(o) = \frac{2}{\pi} n \log n + O(n). \]
\end{lemma}

\begin{proof}
By \eqref{green-origin}, we have 
$f(o)= nG_{n}(o,o) = n (\frac{2}{\pi} \log n + O(1))$, and $|u^n_n(o)-f(o)| = O(n)$ by Lemma~\ref{positive}.
\end{proof}

Turning to the proof of the upper bound in Theorem~\ref{thm:logfraction}, let $N=u_n^n(o)$.
By Lemmas \ref{finiteball} and \ref{monotone}, $I(\rho,N) \le I_n(\rho,N)$.  By Lemma~\ref{inn}, $I_n(\rho,N)=n$.  Now by Lemma~\ref{greenorigin}, $\frac{N}{\log N} = \frac{(2/\pi) n \log n +O(n)}{\log n + O(\log \log n)} = (\frac{2}{\pi} + o(1))n$, hence
	\[ \frac{I(\rho,N)}{N / \log N} \leq \frac{n}{(\frac{2}{\pi}+o(1))n} = \frac{\pi}{2} + o(1). \]
Since $I(\rho,n)$ is nondecreasing in $n$, the desired upper bound follows.

To show the lower bound for $\aligned$ we use lemmas \ref{finiteball} and \ref{an<iun} along with \eqref{manycolumns}  
\[ I{(\aligned,N)} =  \lim_{R\rightarrow \infty}I_R(\aligned,N) \geq  \col(\aligned,n) \geq \beta n = \Theta(\frac{N}{\log N}). \]
Since $I(\rho,n)$ is nondecreasing in $n$ the desired lower bound follows.

\begin{remark}
The proofs of the lower bounds in Theorems~\ref{thm:transient} and~\ref{thm:logfraction} apply to a slightly more general class of rotor configurations than $\aligned$. 
Given a rotor configuration $\rho$, the \emph{forward path} from $x$ is the path $x=x_0, x_1, x_2, \ldots$ defined by $x_{k+1} = x_k + \rho(x_k)$ for $k\geq 0$.  Let us say that $x \in \partial \B_r$ has a \emph{simple path to infinity} if the forward path from $x$ is simple (that is, all $x_k$ are distinct) and $x_k \notin \partial \B_r$ for all $k\geq 1$.
The proofs we have given for~$\aligned$ remain valid for~$\rho$ as long as there is a constant $C$ and a sequence of radii $r_1, r_2, \ldots$ with $r_{i+1}/r_i < C$, such that for each $i$, at least $r_{i}^{d-1}/C$ sites on $\partial \B_{r_{i}}$ have disjoint simple paths to infinity.  For instance, the rotor configuration
	\[ \rho(x) = \begin{cases} \alpha, &x_d \geq 0 \\
						\beta, &x_d <0 \end{cases} \]
satisfies this condition as long as $(\alpha, \beta) \neq (-e_d,+e_d)$.
\end{remark}

\section{Some open questions}

We conclude with a few natural questions. 

\begin{itemize}
\item When is Schramm's bound attained?  In $\Z^d$ for $d\geq 3$ with rotors initially aligned in one direction, is the escape rate for rotor walk asymptotically equal to the escape probability of the simple random walk?  Theorem~\ref{thm:transient} shows that the escape rate is positive.

\item If random walk on a graph is transient, must there be a rotor configuration $\rho$ for which a positive fraction of particles escape to infinity, that is, $\liminf_{n \to \infty} \frac{I(\rho,n)}{n} >0$?

\item Let us choose initial rotors $\rho(x)$ for $x \in \Z^d$ independently and uniformly at random from $\{\pm e_1, \ldots, \pm e_d\}$.
Is the resulting rotor walk recurrent in dimension $2$ and transient in dimensions $d\geq 3$?  Angel and Holroyd \cite[Corollary 6]{angel-holroyd-rec} prove that two initial configurations differing in only a finite number of rotors are either both recurrent or both transient. Hence the set of recurrent $\rho$ is a tail event and consequently has probability $0$ or $1$.

\item Starting from initial rotor configuration $\aligned$ in $\Z^2$, let $\rho_n$ be the rotor configuration after $n$ particles have escaped to infinity.  Does $\rho_n(nx,ny)$ have a limit as $n \to \infty$?  Figure~\ref{f.finalrotors} suggests that the answer is yes.

\item Consider rotor walk in $\Z^2$ with a drift to the north: each rotor mechanism is period $5$ with successive exits cycling through North, North, East, South, West.  Is this walk transient for all initial rotor configurations?
\end{itemize}

Angel and Holroyd resolved many of these questions when $\Z^d$ is replaced by an arbitrary rooted tree: if only finitely many rotors start pointing toward the root (recall we use the prospective convention), then the escape rate for rotor walk started at the root equals the escape probability $\mathcal{E}$ for random walk started at the root \cite[Theorem~3]{angel-holroyd-trees}.  On the other hand if \emph{all} rotors start pointing toward the root, then the rotor walk is recurrent \cite[Theorem 2(iii)]{angel-holroyd-trees}.  On the regular $b$-ary tree, the i.i.d.\ uniformly random initial rotor configuration has escape rate $\mathcal{E}=1/b$ for $b\geq 3$ but is recurrent for $b=2$ \cite[Theorem 6]{angel-holroyd-trees}. In the latter case particles travel extremely far \cite[Theorem~7]{angel-holroyd-trees}: There is a constant $c>0$ such that with probability tending to $1$ as $n \to \infty$, one of the first $n$ particles reaches distance $e^{e^{cn}}$ from the root before returning! 

\section*{Acknowledgment}
This work was initiated while the first three authors were visiting the Theory Group at Microsoft Research Redmond. 

 \bibliographystyle{plain}
  \bibliography{escape-rates-rotor-walks}

\end{document}